\documentclass{amsart}

\title[\KM\ does not imply the class Fodor principle]{Kelley-Morse set theory does not prove\break the class Fodor principle}

\author{Victoria Gitman}
\address[V.~Gitman]{The City University of New York, CUNY Graduate Center, Mathematics Program, 365 Fifth Avenue, New York, NY 10016}
\email{vgitman@nylogic.org}
\urladdr{https://victoriagitman.github.io/}

\author{Joel David Hamkins}
 \address[J.~D.~Hamkins]
         {Professor of Logic, Oxford University \&\ Sir Peter Strawson Fellow, University College, Oxford. High Street Oxford OX1 4BH U.K.}
\email{joeldavid.hamkins@philosophy.ox.ac.uk}
\urladdr{http://jdh.hamkins.org}

\author{Asaf Karagila}
\address[A.~Karagila]{School of Mathematics University of East Anglia. Norwich, NR4~7TJ, UK}
\email{karagila@math.huji.ac.il}
\urladdr{http://karagila.org/}
\thanks{The third author was partially supported by the Austrian Science Foundation
(FWF), Grant~I~3081, and by the Royal Society grant no.~NF170989.}

\thanks{Commentary about this article can be made at \href{http://jdh.hamkins.org/km-does-not-prove-class-fodor}{http://jdh.hamkins.org/km-does-not-prove-class-fodor}.}

\usepackage{latexsym,amsfonts,amsmath,amssymb,mathrsfs}
\usepackage{bbm}
\usepackage{relsize}
\usepackage[hidelinks]{hyperref}
\usepackage[rgb,dvipsnames]{xcolor}
\usepackage{tikz}
\usetikzlibrary{arrows,arrows.meta,petri,topaths,positioning,shapes,shapes.misc,patterns,calc,decorations.pathreplacing,hobby}

\newtheorem{theorem}{Theorem}
\newtheorem*{theorem*}{Theorem}

\newtheorem*{maintheorem*}{Main Theorem}
\newtheorem*{maintheorems*}{Main Theorems}
\newtheorem{corollary}[theorem]{Corollary}
\newtheorem*{corollary*}{Corollary}
\newtheorem*{corollaries*}{Corollaries}

\newtheorem{lemma}[theorem]{Lemma}

\newtheorem{question}[theorem]{Question}
\newtheorem*{question*}{Question}

\newtheorem*{questions*}{Questions}
\newtheorem*{mainquestion*}{Main Question} 
\newtheorem*{openquestion*}{Open Question} 

\newtheorem{definition}[theorem]{Definition}

\newcommand{\QED}{\end{proof}}

\def\proclaim[#1]{{\bf #1}}
\def\BF#1.{{\bf #1.}}

\def\says#1:#2\par{\item[#1] #2\par}

\newcommand{\Los}{\L o\'s}

\newcommand{\Godel}{G\"odel}

\renewcommand{\P}{{\mathbb P}}
\newcommand{\Q}{{\mathbb Q}}

\newcommand{\one}{\mathbbm{1}} 
\makeatletter
\newcommand{\dotminus}{\mathbin{\text{\@dotminus}}}
\newcommand{\@dotminus}{%
  \ooalign{\hidewidth\raise1ex\hbox{.}\hidewidth\cr$\m@th-$\cr}%
}
\makeatother

\newcommand{\of}{\subseteq}

\newcommand{\fo}{\supseteq}

\newcommand{\set}[1]{\{\,{#1}\,\}}
\newcommand{\singleton}[1]{\left\{{#1}\right\}}

\newcommand{\elesub}{\prec}

\newcommand{\dom}{\mathop{\rm dom}}

\newcommand{\Cof}{\mathop{\rm Cof}}

\newcommand{\Add}{\mathop{\rm Add}}

\newcommand{\Con}{\mathop{{\rm Con}}}

\newcommand{\restrict}{\upharpoonright} 
\newcommand{\satisfies}{\models}
\newcommand{\forces}{\Vdash}

\newcommand{\union}{\cup}

\newcommand{\Union}{\bigcup}

\newcommand{\intersect}{\cap}
\newcommand{\Intersect}{\bigcap}

\newcommand{\smalllt}{\mathrel{\mathchoice{\raise2pt\hbox{$\scriptstyle<$}}{\raise1pt\hbox{$\scriptstyle<$}}{\raise0pt\hbox{$\scriptscriptstyle<$}}{\scriptscriptstyle<}}}
\newcommand{\smallleq}{\mathrel{\mathchoice{\raise2pt\hbox{$\scriptstyle\leq$}}{\raise1pt\hbox{$\scriptstyle\leq$}}{\raise1pt\hbox{$\scriptscriptstyle\leq$}}{\scriptscriptstyle\leq}}}

   \def\DHLhksqrt#1#2{%
   \setbox0=\hbox{$#1\sqrt{#2\,}$}\dimen0=\ht0
   \advance\dimen0-0.2\ht0
   \setbox2=\hbox{\vrule height\ht0 depth -\dimen0}%
   {\box0\lower0.4pt\box2}}
\newcommand{\boolval}[1]{\mathopen{\lbrack\!\lbrack}\,#1\,\mathclose{\rbrack\!\rbrack}}
\def\[#1]{\boolval{#1}}
\newbox\gnBoxA
\newbox\gnBoxB
\newdimen\gnCornerHgt
\setbox\gnBoxA=\hbox{\tiny$\ulcorner$}
\global\gnCornerHgt=\ht\gnBoxA
\newdimen\gnArgHgt
\def\gcode #1{%
\setbox\gnBoxA=\hbox{$#1$}%
\setbox\gnBoxB=\hbox{$\bar #1$}%
\gnArgHgt=\ht\gnBoxB%
\ifnum     \gnArgHgt<\gnCornerHgt \gnArgHgt=0pt%
\else \advance \gnArgHgt by -\gnCornerHgt%
\fi \raise\gnArgHgt\hbox{\tiny$\ulcorner$} \box\gnBoxA %
\raise\gnArgHgt\hbox{\tiny$\urcorner$}}
\newcommand{\UnderTilde}[1]{{\setbox1=\hbox{$#1$}\baselineskip=0pt\vtop{\hbox{$#1$}\hbox to\wd1{\hfil$\sim$\hfil}}}{}}
\newcommand{\Undertilde}[1]{{\setbox1=\hbox{$#1$}\baselineskip=0pt\vtop{\hbox{$#1$}\hbox to\wd1{\hfil$\scriptstyle\sim$\hfil}}}{}}
\newcommand{\undertilde}[1]{{\setbox1=\hbox{$#1$}\baselineskip=0pt\vtop{\hbox{$#1$}\hbox to\wd1{\hfil$\scriptscriptstyle\sim$\hfil}}}{}}
\newcommand{\UnderdTilde}[1]{{\setbox1=\hbox{$#1$}\baselineskip=0pt\vtop{\hbox{$#1$}\hbox to\wd1{\hfil$\approx$\hfil}}}{}}
\newcommand{\Underdtilde}[1]{{\setbox1=\hbox{$#1$}\baselineskip=0pt\vtop{\hbox{$#1$}\hbox to\wd1{\hfil\scriptsize$\approx$\hfil}}}{}}

\renewcommand{\implies}{\mathrel{\rightarrow}}

\def\<#1>{\left\langle#1\right\rangle}

\newcommand{\val}{\mathop{\rm val}\nolimits}

\newcommand{\ORD}{\mathord{{\rm Ord}}}
\newcommand{\Ord}{\mathord{{\rm Ord}}}

\newcommand{\ETR}{{\rm ETR}}
\newcommand\ETRord{\ETR_{\Ord}}

\newcommand{\ZFC}{{\rm ZFC}}
\newcommand{\ZF}{{\rm ZF}}

\newcommand{\KM}{{\rm KM}}
\newcommand{\GB}{{\rm GB}}
\newcommand{\GBC}{{\rm GBC}}

\newcommand{\AC}{{\rm AC}}
\newcommand{\DC}{{\rm DC}}
\newcommand{\CC}{{\rm CC}}

\newcommand{\cell}[1]{\boxit{\hbox to 17pt{\strut\hfil$#1$\hfil}}}
\newcommand{\head}[2]{\lower2pt\vbox{\hbox{\strut\footnotesize\it\hskip3pt#2}\boxit{\cell#1}}}
\newcommand{\boxit}[1]{\setbox4=\hbox{\kern2pt#1\kern2pt}\hbox{\vrule\vbox{\hrule\kern2pt\box4\kern2pt\hrule}\vrule}}
\newcommand{\Col}[3]{\hbox{\vbox{\baselineskip=0pt\parskip=0pt\cell#1\cell#2\cell#3}}}
\newcommand{\tapenames}{\raise 5pt\vbox to .7in{\hbox to .8in{\it\hfill input: \strut}\vfill\hbox to
.8in{\it\hfill scratch: \strut}\vfill\hbox to .8in{\it\hfill output: \strut}}}
\newcommand{\Head}[4]{\lower2pt\vbox{\hbox to25pt{\strut\footnotesize\it\hfill#4\hfill}\boxit{\Col#1#2#3}}}
\newcommand{\Dots}{\raise 5pt\vbox to .7in{\hbox{\ $\cdots$\strut}\vfill\hbox{\ $\cdots$\strut}\vfill\hbox{\
$\cdots$\strut}}}
\newcommand{\less}{{<}}

\mathchardef\mhyphen="2D  

\begin{document}

\begin{abstract}
We show that Kelley-Morse set theory does not prove the class Fodor principle, the assertion that every regressive class function $F:S\to\Ord$ defined on a stationary class $S$ is constant on a stationary subclass. Indeed, for every $\omega\leq\lambda\leq\ORD$, it is relatively consistent with \KM\ that there is a class function $F:\Ord\to\lambda$ that is not constant on any stationary class, and moreover $\lambda$ is the least ordinal for which such a counterexample function exists. As a corollary of this result, it is consistent with \KM\ that there is a class $A\of\omega\times\Ord$, such that each section $A_n=\set{\alpha\mid (n,\alpha)\in A}$ contains a class club, but $\Intersect_n A_n$ is empty. Consequently, it is relatively consistent with \KM\ that the class club filter is not $\sigma$-closed.
\end{abstract}

\maketitle

\section{Introduction}\label{Section.Introduction}

The \emph{class Fodor principle} is the assertion that every regressive class function $F:S\to\Ord$ defined on a stationary class $S$ is constant on a stationary subclass of $S$. This statement can be expressed in the usual second-order language of set theory, and the principle can therefore be sensibly considered in the context of any of the various second-order set-theoretic systems, such as \Godel-Bernays \GBC\ set theory or Kelley-Morse \KM\ set theory. Just as with the classical Fodor's lemma in first-order set theory, the class Fodor principle is equivalent, over a weak base theory, to the assertion that the class club filter is normal (see theorem~\ref{th:FodorEquivNormalFilter}). We shall investigate the strength of the class Fodor principle and try to find its place within the natural hierarchy of second-order set theories. We shall also define and study weaker versions of the class Fodor principle.

If one tries to prove the class Fodor principle by adapting one of the classical proofs of the first-order Fodor's lemma, then one inevitably finds oneself needing to appeal to a certain second-order class-choice principle, which goes beyond the axiom of choice and the global choice principle, and which is not available in Kelley-Morse set theory \cite{GitmanHamkins:Kelley-MorseSetTheoryAndChoicePrinciplesForClasses}. For example, in one standard proof, we would want for a given $\ORD$-indexed sequence of non-stationary classes to be able to choose for each member of it a class club that it misses. This would be an instance of class-choice, since we seek to choose classes here, rather than sets. The class choice principle $\CC(\Pi^0_2)$, it turns out, is sufficient for us to make these choices, for this principle states that if every ordinal $\alpha$ admits a class $A$ witnessing a $\Pi^0_2$-assertion $\varphi(\alpha,A)$, allowing class parameters, then there is a single class $B\subseteq \Ord\times V$, whose slices $B_\alpha$ witness $\varphi(\alpha,B_\alpha)$; and the property of being a class club avoiding a given class is $\Pi^0_2$ expressible.

Thus, the class Fodor principle, and consequently also the normality of the class club filter, is provable in the relatively weak second-order set theory $\GBC+\CC(\Pi^0_2)$, as shown in theorem \ref{Theorem.CC-implies-Fodor}. This theory is known to be weaker in consistency strength than the theory $\GBC+\Pi^1_1$-comprehension, which is itself strictly weaker in consistency strength than \KM. But meanwhile, although the class choice principle is weak in consistency strength, it is not actually provable in \KM; indeed, even the weak fragment $\CC(\Pi^0_1)$ is not provable in \KM. Those results were proved several years ago by the first two authors \cite{GitmanHamkins:Kelley-MorseSetTheoryAndChoicePrinciplesForClasses}. In light of those results, however, one should perhaps not have expected to be able to prove the class Fodor principle in \KM.

Indeed, it follows similarly from arguments of the third author in \cite{Karagila2018:Fodors-lemma-can-fail-everywhere} that if $\kappa$ is an inaccessible cardinal, then there is a forcing extension $V[G]$ with a symmetric submodel $M$ such that $V_\kappa^M=V_\kappa$, which implies that $\mathcal M=(V_\kappa,\in, V^M_{\kappa+1})$ is a model of Kelley-Morse, and in $\mathcal M$, the class Fodor principle fails in a very strong sense.


\begin{maintheorem*}
Kelley-Morse set theory \KM, if consistent, does not prove the class Fodor principle. Indeed, every model of $\KM$ can be extended to a model of $\KM$ with the same first-order part, but with a class function $F:\ORD\to \omega$, which is not constant on any stationary class; in this model, therefore, the class club filter is not $\sigma$-closed.
\end{maintheorem*}
\noindent The proof is a combination of the ideas from \cite{Karagila2018:Fodors-lemma-can-fail-everywhere} and the methods of \cite{GitmanHamkins:Kelley-MorseSetTheoryAndChoicePrinciplesForClasses}, as well as new results on class club shooting.

We shall also investigate various weak versions of the class Fodor principle.

\begin{definition}\
\textup{\begin{enumerate}
\item For a cardinal $\kappa$, the \emph{class $\kappa$-Fodor principle} asserts that every class function $F:S\to\kappa$ defined on a stationary class $S\of\Ord$ is constant on a stationary subclass of $S$.
\item The \emph{class $\less\Ord$-Fodor principle} is the assertion that the $\kappa$-class Fodor principle holds for every cardinal $\kappa$.
\item The \emph{bounded class Fodor principle} asserts that every regressive class function $F:S\to\Ord$ on a stationary class $S\of\Ord$ is bounded on a stationary subclass of $S$.
\item The \emph{very weak class Fodor principle} asserts that every regressive class function $F:S\to\Ord$ on a stationary class $S\of\Ord$ is constant on an unbounded subclass of $S$.
\end{enumerate}}
\end{definition}

We shall separate these principles as follows.

\begin{theorem} Suppose $\KM$ is consistent.
\begin{enumerate}
\item There is a model of $\KM$ in which the class Fodor principle fails, but the class $\less\Ord$-Fodor principle holds.
\item There is a model of $\KM$ in which the class $\omega$-Fodor principle fails, but the bounded class Fodor principle holds.
\item There is a model of $\KM$ in which the class $\omega$-Fodor principle holds, but the bounded class Fodor principle fails.
\item $\GB^-$ proves the very weak class Fodor principle.
\end{enumerate}
\end{theorem}

Finally, we show that the class Fodor principle can be neither created nor destroyed by set forcing.

\begin{theorem}
 The class Fodor principle is invariant by set forcing over models of $\GBC^-$. That is, it holds in an extension if and only if it holds in the ground model.
\end{theorem}

Let us conclude the introduction by mentioning the following easy negative instance of the class Fodor principle for certain \GBC\ models. This argument seems to be a part of set-theoretic folklore. Namely, consider an $\omega$-standard model of \GBC\ set theory $M$ having no $V_\kappa^M$ that is a model of \ZFC. A minimal transitive model of \ZFC, for example, has this property. Inside $M$, let $F(\kappa)$ be the least $n$ such that $V_\kappa^M$ fails to satisfy $\Sigma_n$-collection. This is a definable class function $F:\Ord^M\to\omega$ in $M$, but it cannot be constant on any stationary class in $M$, because by the reflection theorem, for every $n<\omega$, there is a class club of cardinals $\kappa$ such that $V_\kappa^M$ satisfies $\Sigma_n$-collection.

\section{Second-order set theories}\label{sec:second-order}

We shall express all our second-order set-theories in a two-sorted language, with separate variables for sets and classes, and a binary relation $\in$ used both for the membership relation on sets $a\in b$ and also for membership of a set in a class $a\in A$. Models of these theories, in the Henkins semantics, can be taken to have the form $\mathcal M=\< M,\in^{\mathcal M},\mathscr C>$, where $M$ is the collection of sets in the model and $\mathscr C$ is the collection of classes $A$ with $A\subseteq M$.

Let us briefly review the principal second-order set theories we shall consider. \Godel-Bernays set theory \GB\ asserts $\ZF$ for the first-order part of the structure, plus extensionality for classes, the replacement axiom for class functions, and the class-comprehension axiom, which asserts that every first-order formula, allowing set and class parameters, defines a class. The theory $\GB+\AC$ includes the axiom of choice for sets, and the theory \GBC\ includes the global choice principle, which asserts that there is a well-order of the sets in type $\Ord$.
The natural weakenings $\GB^-$ and $\GBC^-$ drop the power set axiom, asserting only $\ZF^-$ for the first-order part (when dropping the power set axiom, be sure to use the collection and separation schemes in $\ZF^-$, instead of just replacement, in light of \cite{GitmanHamkinsJohnstone2016:WhatIsTheTheoryZFC-Powerset?}).

A hierarchy of theories grows above \GBC\ by strengthening the class comprehension axiom into the second-order language. The theory $\GBC+\Pi^1_n$-comprehension, for example, allows class comprehension for $\Pi^1_n$-expressible properties, with class parameters. Growing in power as $n$ increases, this hierarchy culminates ultimately in Kelley-Morse set theory \KM, which is \GBC\ together with the full second-order class comprehension scheme.

Recent work has revealed a central role in the hierarchy for the principle of elementary transfinite recursion \ETR, which asserts that every first-order definable class recursion along a class well-order, including orders longer than $\Ord$, has a solution (see \cite{GitmanHamkins2016:OpenDeterminacyForClassGames, GitmanHamkinsHolySchlichtWilliams:The-exact-strength-of-the-class-forcing-theorem}). The principle is naturally stratified according to the length of the class well order. The principle $\ETRord$, for example, asserting solutions for class recursions of length $\Ord$, is equivalent to the assertion that every class forcing notion admits a forcing relation \cite{GitmanHamkinsHolySchlichtWilliams:The-exact-strength-of-the-class-forcing-theorem}; and the principle $\ETR_\omega$, for class recursions of length $\omega$, is equivalent to the assertion that every structure $\<V,\in,A>$, for any class $A$, admits a first-order truth predicate; the Tarskian recursive definition of truth is, after all, a recursion of length $\omega$ on formulas\footnote{The converse direction was pointed out to the second author by Kentaro Fujimoto in a private communication.}. In this way, the restricted $\ETR$ principles span the region between $\GBC$ and $\GBC+\Pi^1_1$-comprehension, which proves $\ETR$.

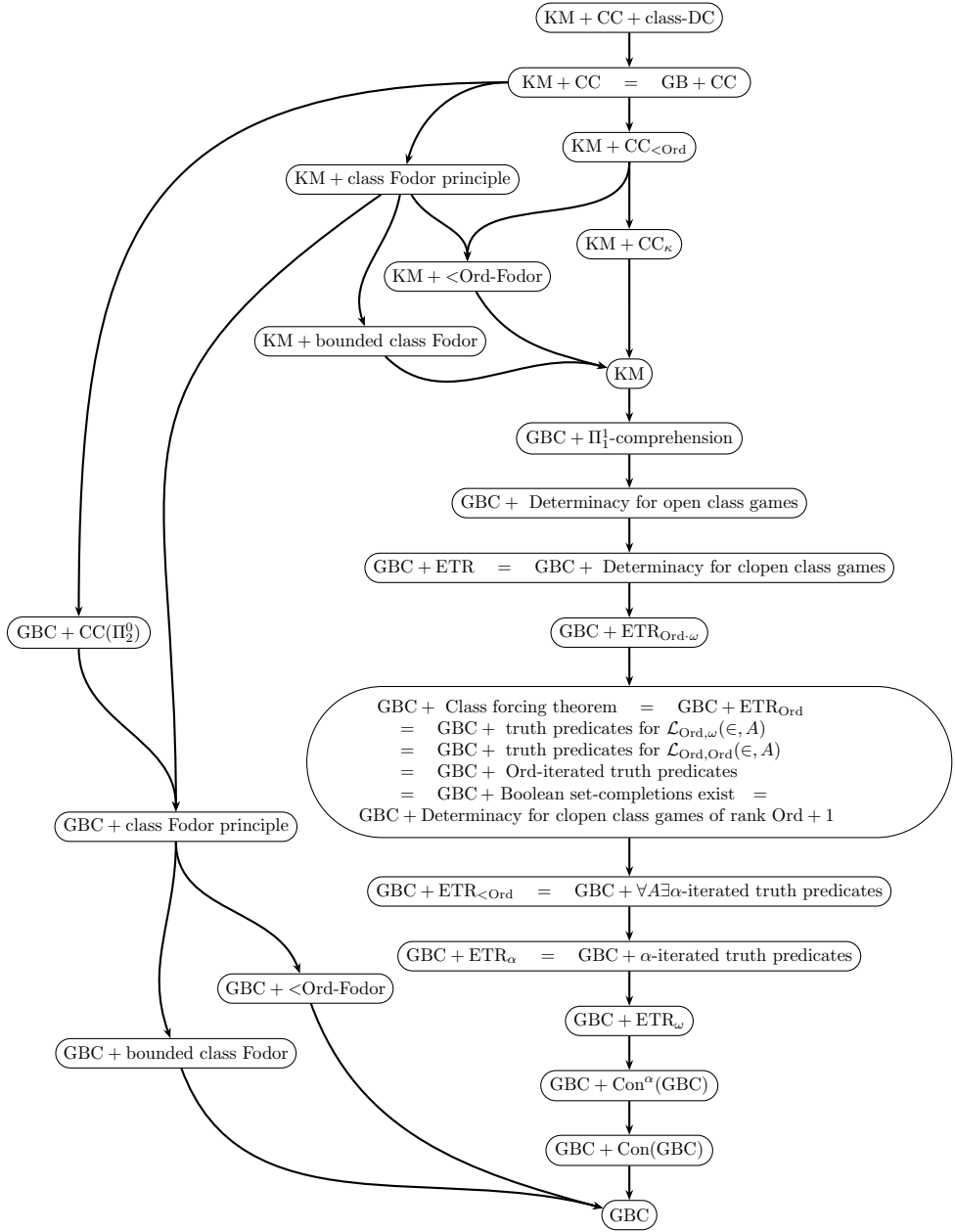
\begin{figure}[h]
  \begin{tikzpicture}[theory/.style={draw,rounded rectangle,scale=.7,minimum height=5.7mm},scale=.44]
     \draw (0:0) node[theory] (GBC) {$\GBC$}
           ++(90:2) node[theory] (ConZFC) {$\GBC+\Con(\GBC)$}
           ++(90:2) node[theory] (ConalphaZFC) {$\GBC+\Con^\alpha(\GBC)$}
           ++(90:2) node[theory] (ETRomega) {$\GBC+\ETR_\omega$}
           ++(90:2) node[theory] (ETRalpha) {$\GBC+\ETR_\alpha\quad=\quad\GBC+\alpha\text{-iterated truth predicates}$}
           ++(90:2) node[theory] (ETR<Ord) {$\GBC+\ETR_{<\Ord}\quad=\quad \GBC+\forall A\exists \alpha\text{-iterated truth predicates}$}
           ++(90:4) node[theory] (ETRord) {\parbox[c][2.7cm]{10.5cm}{{\quad}$\GBC+\text{ Class forcing theorem}\quad=\quad \GBC+\ETRord$\\
            ${\qquad}=\quad\GBC+\text{ truth predicates for }\mathcal{L}_{\Ord,\omega}(\in,A)$\\
            ${\qquad}=\quad\GBC+\text{ truth predicates for }\mathcal{L}_{\Ord,\Ord}(\in,A)$\\
            ${\qquad}=\quad\GBC+\text{ $\Ord$-iterated truth predicates }$\\
            ${\qquad}=\quad\GBC\,+$ Boolean set-completions exist\quad $=$ \\
            $\GBC+\text{Determinacy for clopen class games of rank }\Ord+1$}}
           ++(90:4) node[theory] (ETROrd*omega) {$\GBC+\ETR_{\Ord\cdot\omega}$}
           ++(90:2) node[theory] (ETR) {$\GBC+\ETR\quad=\quad\GBC+\text{ Determinacy for clopen class games}$}
           ++(90:2) node[theory] (Open) {$\GBC+\text{ Determinacy for open class games}$}
           ++(90:2) node[theory] (Pi11) {$\GBC+\Pi^1_1$-comprehension}
           ++(90:2) node[theory] (KM) {$\KM$}
           ++(90:4) node[theory] (KMkappa-choice) {$\KM+\CC_\kappa$}
           ++(90:3) node[theory] (KM<Ord-choice) {$\KM+\CC_{{<}\Ord}$}
           ++(90:2) node[theory] (KM+CC) {\ $\KM+\CC\quad=\quad\GB+\CC$ \ }
           ++(90:2) node[theory] (KM+DC) {$\KM+\CC+\text{class-DC}$}
           (-17,18) node[theory] (GBCC) {$\GBC+\CC(\Pi^0_2)$}
           (-14,12) node[theory] (GBC+Fodor) {$\GBC+\text{class Fodor principle}$}
           (-14,5) node[theory] (GBC+BoundedFodor) {$\GBC+\text{bounded class Fodor}$}
           (-10,7) node[theory] (GBC+<Ord-Fodor) {$\GBC+\less\Ord$-Fodor}
           (-7,32) node[theory] (KM+Fodor) {$\KM+\text{class Fodor principle}$}
           (-8,27) node[theory] (KM+BoundedFodor) {$\KM+\text{bounded class Fodor}$}
           (-5,29) node[theory] (KM+<Ord-Fodor) {$\KM+\less\Ord$-Fodor};
     \draw[{[scale=.6]<}-,>=Stealth,thick]
                        (GBC) edge (ConZFC)
                        (ConZFC) edge (ConalphaZFC)
                        (ConalphaZFC) edge (ETRomega)
                        (ETRomega) edge (ETRalpha)
                        (ETRalpha) edge (ETR<Ord)
                        (ETR<Ord) edge (ETRord)
                        (ETRord) edge (ETROrd*omega)
                        (ETROrd*omega) edge (ETR)
                        (ETR) edge (Open)
                        (Open) edge (Pi11)
                        (Pi11) edge (KM)
                        (KM) edge (KMkappa-choice)
                        (KMkappa-choice) edge (KM<Ord-choice)
                        (KM<Ord-choice) edge (KM+CC)
                        (KM+CC) edge (KM+DC)
                        (KM+Fodor) edge[in=180,out=75] (KM+CC)
                        (GBC+Fodor) edge[in=-145,out=90,looseness=1.4] (KM+Fodor)
                        (GBCC) edge[in=180,out=90,looseness=1.4] (KM+CC)
                        (GBC+Fodor) edge[in=-90,out=90] (GBCC)
                        (GBC) edge[out=165,in=-70] (GBC+<Ord-Fodor)
                        (GBC) edge[out=165,in=-70] (GBC+BoundedFodor)
                        (GBC+<Ord-Fodor) edge[out=110,in=-90] (GBC+Fodor)
                        (GBC+BoundedFodor) edge[out=110,in=-90] (GBC+Fodor)
                        (KM) edge[out=165,in=-60] (KM+<Ord-Fodor)
                        (KM) edge[out=165,in=-45] (KM+BoundedFodor)
                        (KM+<Ord-Fodor) edge[out=90,in=-60] (KM+Fodor)
                        (KM+<Ord-Fodor) edge[out=90,in=-90] (KM<Ord-choice)
                        (KM+BoundedFodor) edge[out=110,in=-100] (KM+Fodor);
  \end{tikzpicture}
  \caption{The hierarchy of second-order set theories}\label{Figure.Hierarchy-of-second-order-theories}
\end{figure}

Next, let's introduce the class choice principles. An enumerated collection of classes $\<X_a\mid a\in A>$ indexed by a class $A$ in a model of second-order set theory is \emph{coded} in the model, if there is a class $X\subseteq A\times V$, whose slices $X_a=\set{x\mid (a,x)\in X}$ are exactly the corresponding classes on the sequence.

The \emph{class choice} principle \CC\ is the assertion that for any class $A$ and class parameter $Z$, if for every $a\in A$ there is a class $X$ with $\varphi(a,X,Z)$, where $\varphi$ is any assertion in the second-order language of set theory, then there is a coded sequence of classes choosing such witnesses; that is, there is a class $X\of A\times V$, such that $\varphi(a,X_a,Z)$ for every $a\in A$. Stratifying this axiom scheme, let $\CC(\Pi^0_n)$ be the principle resulting when we allow $\varphi$ only of complexity $\Pi^0_n$, and similarly for other levels of complexity. The axiom $\CC_\kappa$ is the case $A=\kappa$.

The first and second authors showed in \cite{GitmanHamkins:Kelley-MorseSetTheoryAndChoicePrinciplesForClasses} that Kelley-Morse set theory does not prove class choice, even for $\Pi^0_1$ assertions, in fact, it does not even prove the principle $\CC_\omega(\Pi^0_1)$. Meanwhile, the theories $\KM$ and $\KM+\CC$ are equiconsistent, as shown by Marek \cite{Marek:KM}, because the `constructible universe' of a $\KM$-model (where one not only restricts the sets to those in the constructible hierarchy up to $\Ord$, but also restricts the classes to those with codes appearing in the constructible hierarchy at some coded meta-ordinal stage above $\Ord$ in the corresponding unrolled structure obtained from the model) satisfies $\KM+\CC$. This equivalence holds level-by-level, so that $\GBC+\Pi^1_n$-comprehension is equiconsistent with $\GBC+\Pi^1_n$-comprehension together with $\CC(\Pi^1_n)$. Similarly, Kameryn Williams \cite{williams:dissertation} proved that the consistency of $\GBC+\ETR$ implies the consistency of $\GBC$ together with class choice for first-order assertions, thereby providing an upper bound on the consistency strength of $\GBC$ with the class Fodor principle; his dissertation \cite{williams:dissertation} is an excellent resource for all these matters. Note that because $\GBC+\CC(\Pi^0_2)$ proves the class Fodor principle, while \KM\ does not, our main theorem provides natural instances of non-linearity in the hierarchy of second-order set theories.

Figure \ref{Figure.Hierarchy-of-second-order-theories} shows the implication diagram for the hierarchy of second-order theories (except that there should be additional arrows, which we omitted for clarity, from $\KM+\text{bounded class Fodor}$ and $\KM+\less\Ord$-Fodor to the corresponding respective $\GBC$ theories at the bottom left).  Many of the arrows in the diagram, but not all, are also strict consistency strength implications. For example, the \ETR\ principles rise in consistency strength. Yet, we've mentioned that $\KM+\CC$ is equiconsistent with \KM; it follows that all the various class Fodor principles are equiconsistent over \KM. We will show in Corollary~\ref{cor:GBCEquiconsistent} that it is also the case that the class Fodor principles are equiconsistent over $\GBC$. So, in particular, the class Fodor principle does not have any strength. In terms of consistency strength, the entire diagram does not vary much, since even $\KM+\CC+\text{class-}\DC$\footnote{Class-$\DC$ is a version of dependent choice for classes asserting that any definable relation on classes without terminal nodes has a branch of length $\omega$.} at the top is strictly below $\ZFC+$ one inaccessible cardinal. The diagram contains numerous second-order principles of interest, whose meaning we will not explain in this article, but which we felt should be added for completeness. The interested reader can find a discussion of these various principles in \cite{ GitmanHamkinsHolySchlichtWilliams:The-exact-strength-of-the-class-forcing-theorem} and \cite{AntosFriedmanGitman:booleanclassforcing}.

We say that a class $A\of\Ord$ is in the \emph{class club filter}, if $A$ contains a class club. The filter itself is not a class, nor even coded as a class; it is nevertheless sensible to refer to the filter in our models by viewing it as a second-order definable meta-class, a predicate on the classes. The class club filter is ${\leq}\alpha$-\emph{closed} for an ordinal $\alpha$, if whenever $\< A_\xi\mid\xi<\alpha>$ is a coded collection of classes $A_\xi\of\Ord$, each in the class club filter, then the intersection $\Intersect_\xi A_\xi$ is also in the class club filter. This is expressible in the second-order language of set theory. Similarly, the class club filter is \emph{normal}, if whenever $\<A_\xi\mid\xi\in\ORD>$ is a coded collection of classes in the class club filter, then its diagonal intersection is also in the class club filter.\goodbreak

\section{Positive instances of the class Fodor principle}

Let us begin with the observation that $\GB^-$, without any choice principle at all, proves the very weak class Fodor principle. A version of this theorem was observed by Walter Neumer in 1951 \cite{Neumer1951:Verallgemeinerung-eines-Satzes-von-Alexandroff-und-Urysohn}.

\begin{theorem}[The very weak class Fodor principle]\label{th:weakFodor}
  Assume $\GB^-$. Then every regressive class function $F:S\to\Ord$ defined on a stationary class $S$ is constant on an unbounded subclass $A\of S$.
\end{theorem}

\begin{proof}
Assume that $S$ is a stationary class of ordinals and $F(\gamma)<\gamma$ for all $\gamma\in S$. If $F$ is not constant on any unbounded subclass of $S$, then for every ordinal $\gamma$, there is the (least) ordinal bound $\beta_\gamma$ on the pre-image of $\gamma$ under $F$. That is, $F(\alpha)\neq\gamma$ for all $\alpha\geq\beta_\gamma$. The function $\gamma\mapsto\beta_\gamma$ is first-order definable from $F$ and therefore forms a class in $\GB^-$. Let $C$ be the class of closure points of this function, so $\theta\in C$ if and only if $\beta_\gamma<\theta$ for every $\gamma<\theta$. This is a closed unbounded class of ordinals. Since $S$ is stationary, there is some $\theta\in C\intersect S$. But notice that $F(\theta)\neq\gamma$ for every $\gamma<\theta$, since $\beta_\gamma<\theta$ for all such $\gamma$. This contradicts our assumption that $F$ is regressive on $S$.
\end{proof}

It follows that for any regressive function $F:S\to\Ord$ defined on a stationary class $S$ and any coded sequence of closed unbounded classes $\<C_\alpha\mid\alpha\in\Ord>$, the function $F$ has constant value $\gamma$ on a subclass $T\of S$ such that $T\intersect C_\alpha$ is unbounded for every $\alpha$. This follows from the above theorem because the diagonal intersection $C=\triangle_\alpha C_\alpha$ is club and $F$ is regressive on $S\intersect C$, which is itself a stationary class, and so there is an unbounded class $T\of S\intersect C$ on which $F$ is constant.

Note that one cannot omit or even weaken the assumption that $S$ is stationary in theorem \ref{th:weakFodor}, because if $S$ is not stationary, then there is a closed unbounded class of ordinals $C\of\Ord$ disjoint from $S$, and in this case the function $F(\gamma)=\sup(C\intersect\gamma)$ is regressive on $S$, but not constant on any unbounded class.

The proof of theorem~\ref{th:weakFodor} can similarly be used to show in $\ZF^-$ that every regressive function on a regular cardinal is constant on an unbounded subset.

We observe next that the classical equivalence between Fodor's lemma and the normality of the class club filter translates to the class version.

\begin{theorem}\label{th:FodorEquivNormalFilter}
 The following are equivalent over $\GB^-$.
 \begin{enumerate}
   \item The class Fodor principle.
   \item The class club filter is normal.
 \end{enumerate}
\end{theorem}

\begin{proof}
($1\to 2$) Assume that the class Fodor principle holds, and suppose that we have a coded $\Ord$-sequence $\<A_\alpha\mid\alpha<\Ord>$ of classes $A_\alpha$, each of which contains a class club. Note that we cannot necessarily choose these clubs, since $\GB^-$ and, indeed, even \KM\ does not prove the class-choice principle \CC. But we can appeal to the class Fodor principle. If the diagonal intersection $A=\triangle_\alpha A_\alpha$ does not contain a class club, then the complement $S=\Ord-A$ is stationary. For $\beta\in S$, let $F(\beta)$ be the least $\alpha<\beta$ for which $\beta\notin A_\alpha$. So $F:S\to\Ord$ is a regressive class function on the stationary class $S$. If $F\restrict T$ has constant value $\alpha$ for some $T\of S$, then $T$ is disjoint from $A_\alpha$, and hence disjoint from a class club, and therefore not stationary. So $F$ is not constant on any stationary class. Therefore, the diagonal intersection $\triangle_\alpha A_\alpha$ must contain a class club after all, and so the class club filter is closed under diagonal intersection.

($2\to 1$) Conversely, suppose that the class club filter is closed under diagonal intersections and that $F:S\to\Ord$ is a regressive class function defined on a stationary class $S\of\Ord$. Let $B_\alpha=F^{-1}(\{\alpha\})=\set{\beta\mid F(\beta)=\alpha}$. If some $B_\alpha$ is stationary, then $F$ is constant on a stationary class, and we're done. Otherwise, the complements $A_\alpha=\Ord-B_\alpha$ each contain a class club. In this case, it follows by our assumption that the diagonal intersection $A=\triangle_\alpha A_\alpha$ also contains a class club. But note that $\beta\in A\intersect S$ implies $F(\beta)\neq\alpha$ for any $\alpha<\beta$, contrary to the fact that $F$ was regressive on $S$.
\end{proof}


\begin{theorem}
 The following are equivalent in $\GB^-$ for any ordinal $\kappa$.
 \begin{enumerate}
   \item The class $\kappa$-Fodor principle.
   \item The class club filter is ${\leq}\kappa$-closed.
 \end{enumerate}
\end{theorem}

The proof is analogous to the proof of theorem~\ref{th:FodorEquivNormalFilter}, if we everywhere replace diagonal intersections of length $\Ord$ by intersections of length $\kappa$.

Finally, as we mentioned in the introduction, we can implement one of the proofs of the classical Fodor's lemma to prove the class Fodor principle, if we assume a sufficient class-choice principle.

\begin{theorem}\label{Theorem.CC-implies-Fodor}
The theory $\GB^-$ augmented with the class-choice principle $\CC(\Pi^0_2)$ proves the class Fodor principle. More generally, $\GB^-+\CC_\kappa(\Pi^0_2)$ proves the class $\kappa$-Fodor principle.
\end{theorem}

\begin{proof}
Suppose toward contradiction that $F:S\to\Ord$ is regressive on a stationary class $S\of\Ord$, but not constant on any stationary subclass of $S$. So for each ordinal $\alpha$, there is a class club $C$ with $F(\gamma)\neq\alpha$ for all $\gamma\in C$. This is a $\Pi^0_2$ property of the class $C$, with class parameter $F$. By the class-choice principle $\CC(\Pi^0_2)$, therefore, we may choose such clubs, and so there is a coded sequence of classes $\<C_\alpha\mid \alpha\in\Ord>$, where $\gamma\in C_\alpha\implies F(\gamma)\neq\alpha$. Let $C=\triangle_{\alpha\in\Ord}C_\alpha$ be the diagonal intersection of these class clubs, which is itself a class club and therefore meets $S$. So there is some $\gamma\in C\intersect S$. But since $F(\gamma)\neq\alpha$ for all $\alpha<\gamma$, this contradicts our assumption that $F$ was regressive on $S$.

The same argument works for $F:S\to\kappa$, using only $\CC_\kappa(\Pi^0_2)$, since we need only pick $\kappa$ many clubs in this case.
\end{proof}

We will show in Section~\ref{sec:separatingPrinciples} that it is consistent relative to $\KM$ that the class $\less\Ord$-Fodor principle holds, but the class Fodor principle fails.


\begin{lemma}\label{lem:FodorEquivLessOrdFodor}
The class Fodor principle is equivalent over $\GB^-$ to the conjunction of the bounded class Fodor principle and the class $\less\Ord$-Fodor principle.
\end{lemma}

\begin{proof}
The class Fodor principle clearly implies both the bounded class Fodor principle and the class $\less\Ord$-Fodor principle. Conversely, if $F:S\to \Ord$ is regressive on a stationary class, then by the bounded class Fodor principle, we can first restrict to a stationary class $T\of S$ on which $F$ is bounded by some ordinal $\beta$, and then by the $\beta$-Fodor principle, we can further restrict to a stationary subclass on which $F$ is constant, as desired.
\end{proof}

Consequently, once the bounded class Fodor principle holds, then the full class Fodor principle is equivalent to the class $\less\Ord$-Fodor principle. It follows from the theorems in section \ref{sec:separatingPrinciples} that neither the bounded class Fodor principle nor the class $\less\Ord$-Fodor principle implies the other.

\section{Class forcing in second-order set theories}\label{section.second-order-forcing}

Our main argument will involve class forcing over models of second-order set theory, and so let us briefly review some of the basics.

A class forcing notion $\P$ is simply a class partial pre-order, with maximal element $\one$, and a \emph{$\P$-name} is a set or class consisting of pairs $(\tau,p)$, where $\tau$ is a set $\P$-name and $p\in\P$. The check names $\check x=\set{\<\check y,\one>\mid y\in x}$, for example, are defined recursively, and similarly for classes $\check A=\set{\<\check a,\one>\mid a\in A}$.

A class forcing notion $\P$ admits a \emph{forcing relation}, if there is a relation $\forces_\P$ satisfying the \emph{forcing relation recursion}, as described in \cite{GitmanHamkinsHolySchlichtWilliams:The-exact-strength-of-the-class-forcing-theorem}. This is a class recursion on $\P$-names, by which the forcing relation $p\forces_\P\sigma=\tau$ and $p\forces_\P\sigma\in\tau$ on atomic formulas is seen to obey the expected recursive properties. This recursion is expressible internally in the language of second-order set theory, making no reference to countable models, generic filters or the truth and definability lemmas. One easily extends the forcing relation from atomic formulas, when it exists, to forcing relations on other formulas---the atomic forcing relation is the key difficult case.

The theory \GBC, it turns out, is insufficient to prove that every class forcing notion admits a forcing relation, although the stronger theory $\GBC+\ETR$ can prove this, and indeed, the main theorem of \cite{GitmanHamkinsHolySchlichtWilliams:The-exact-strength-of-the-class-forcing-theorem} is that the existence of forcing relations for all class forcing is equivalent over \GBC\ to the principle $\ETRord$.

Meanwhile, \GBC\ is able to prove that forcing relations exist for certain particularly well-behaved forcing notions, including all the forcing we shall employ in this article. For instance, in $\GB^-$ every pre-tame forcing notion, and consequently also every tame forcing notion (see \cite{friedman:classforcing} for definitions), admits a definable forcing relation. Such forcing preserves both $\KM$ and $\KM+\CC$ to the forcing extension.

\begin{theorem}[Stanley, see \cite{PeterHolyRegulaKrapfPhilippSchlicht:classforcing2}]\label{theorem.Tame-forcing-has-definable-forcing-relation}
Assume $\GB^-$. Every pre-tame class forcing notion $\P$ admits a first-order $\P$-definable forcing relation for atomic formulas (and hence forcing relations for any particular formula).
\end{theorem}

\begin{theorem}[\cite{antos:ClassForcingInClassSetTheory}]
Tame forcing preserves $\KM$ to forcing extensions.
\end{theorem}

\begin{theorem}[\cite{AntosFriedmanGitman:booleanclassforcing}]\label{theorem.KM+CC-is-preserved-to-forcing-extensions}
Tame forcing preserves $\KM+\CC$ to forcing extensions.
\end{theorem}
\begin{proof}
Suppose $\mathcal V=\langle V,\in,\mathscr C\rangle\models\KM+\CC$. Let $\mathcal V[G]=\langle V[G],\in,\mathscr C[G]\rangle$ be a forcing extension by a tame class forcing $\P\in\mathscr C$. The tameness of $\P$ implies that $\mathcal V[G]$ is a model of $\KM$, so it remains to verify class choice $\CC$.

Using class choice, we will argue that whenever $p$ is a condition in $\P$ and \hbox{$p\forces\exists X\varphi(X,\dot A)$}, then there is a $\P$-name $\dot X$ such that $p\forces \varphi(\dot X,\dot A)$. Let $D$ be the dense class of conditions $q$ below $p$ for which there is a class $\P$-name $\dot X_q$ such that $q\forces \varphi(\dot X_q,\dot A)$. The class $D$ exists by comprehension. Let $A$ be a maximal antichain of $D$. Now, using class choice, we can pick for every $q\in A$, a class $\P$-name $\dot X_q$ such that $q\forces\varphi(\dot X_q,\dot A)$. After this, we do the usual mixing argument to build the name $\dot X$, that is $\dot X=\Union_{q\in A}\{(\tau,r)\mid r\leq q, r\forces \tau\in \dot X_q,\tau\in\text{dom}(\dot X_q)\}$.

Now suppose that $\mathcal V[G]$ satisfies $\forall \alpha\exists X\varphi(\alpha,X,A)$. So there is a condition $p\in G$ such that $p\forces \forall \alpha\exists X\varphi(\alpha,X,\dot A)$, where $\dot A$ is a class $\P$-name for $A$. Fix an ordinal $\alpha$. By the argument above, we can build a class $\P$-name $\dot X_\alpha$ such that $p\forces\varphi(\check\alpha,\dot X_\alpha,\dot A)$. Again using class choice, we pick for every $\alpha$, a class $\P$-name $\dot X_\alpha$ and put them all together to form a $\P$-name for the sequence of choices in $\mathcal V[G]$.
\end{proof}

For the following comments, we strengthen our theory to $\GBC^-$, adding the global choice axiom (actually just the axiom of choice suffices).\footnote{Indeed, in the absence of the axiom of choice it is consistent that an $\less\ORD$-distributive forcing adds sets. For an explicit example, see Theorem~5.1 in \cite{SchlichtKaragila}.} Strengthening tameness, a class forcing notion $\P$ is \emph{$\less\Ord$-distributive}, if for any ordinal $\beta$ and any coded $\beta$-sequence $\<D_\alpha\mid\alpha<\beta>$ of open dense classes $D_\alpha\of\P$, the intersection $\Intersect_{\alpha<\beta}D_\alpha$ is dense in $\P$. It is an immediate consequence of $\less\Ord$-distributivity that such forcing is pre-tame, and consequently, such forcing admits a definable forcing relation. The standard distributivity arguments then show that such forcing adds no new sets, and so it is also tame.

Strengthening further, we note that if a forcing notion has $\kappa$-closed dense subclasses for every ordinal $\kappa$, then it follows easily that it is $\less\Ord$-distributive, and consequently it admits a definable forcing relation and adds no new sets.

Typically, the forcing language for forcing over models of second-order set theory includes the membership relation $\in$ and constants for every $\P$-name, for both sets and classes. The case of $\check V$ amounts to a predicate for the ground model sets, where $p\forces \sigma\in\check V$ just in case it is dense below $p$ that a condition forces $\sigma=\check x$ for some set $x$ in $V$.

In this article, we should like to augment the forcing language with a predicate $\check{\mathscr C}$ also for the ground-model classes, defining for a class name $\dot X$ that $p\forces\dot X\in\check{\mathscr C}$ just in case it is dense below $p$ to force $\dot X=\check Y$ for some class $Y\in\mathscr C$. Note that this definition involves second-order quantifiers, which will affect the complexity of forced second-order assertions making use of the $\check{\mathscr C}$ class predicate. But in models of \KM, as in our application, this will cause no problem.

In the proof of our main theorem, we shall make a certain homogeneity argument concerning our main forcing notion. Let us therefore define here that a forcing notion $\P$ is \emph{locally homogeneous}, if for every pair of conditions $p$ and $p'$, there are stronger conditions $q\leq p$ and $q'\leq p'$ such that $\P\restrict q\cong \P\restrict q'$. By carrying the automorphism through the forcing relation, just as in the familiar case of set forcing, it follows from this that if a condition $p$ forces $\varphi(\check a,\check A)$, a statement involving only check names, then every condition $q\in\P$ forces $\varphi(\check a,\check A)$.\goodbreak

Finally, we describe how to construct the forcing extension models. Suppose that $\P$ is a class forcing notion in the model of second-order set theory $\mathcal W=\<W,\in^{\mathcal W},\mathscr C>$, which has a forcing relation for it. A filter $G\of\P$ is $\mathcal W$-generic, if $G\intersect D$ is nonempty for every dense class $D\of\P$ in $\mathscr C$. Every countable model, for example, even an ill-founded model, is easily seen to admit such generic filters.

Define the equivalence relation $\sigma=_G\tau$ to hold when there is some $p\in G$ with $p\forces\sigma=\tau$; and the membership relation $\sigma\in_G\tau$, when some $p\in G$ has $p\forces \sigma\in\tau$. Similarly define the class relations $\dot X=_G\dot Y$ and $\sigma\in_G\dot X$. The forcing extension
${\mathcal W}[G]=\langle W[G],\in^{{\mathcal W}[G]},{\mathscr C}[G]\rangle$ is then formed as equivalence classes of set and class names by the relation $=_G$, which is a congruence with respect to $\in_G$. When the model is transitive or at least well-founded, one may equivalently view this construction as proceeding by a recursive interpretation-of-names definition of values $\sigma_G=\val(\sigma,G)=\set{\val(\tau,G)\mid \exists p\in G\ \<\tau,p>\in\sigma}$, but in the general case, that recursion is not actually sensible. One can augment the structure with the ground-model predicates $\check V$ and $\check{\mathscr C}$, as we discussed above.

Finally, one proves what amounts to a forcing analogue of the \Los-theorem, namely, ${\mathcal W}[G]\satisfies\varphi$ just in case there is some $p\in G$ with $p\forces\varphi$, fulfilling the desired alignment between what is forced and what is true in the forcing extension. See further details in \cite[\S 3]{GitmanHamkinsHolySchlichtWilliams:The-exact-strength-of-the-class-forcing-theorem}.

\bigskip
\section{The class-club-shooting forcing}

We shall now introduce and study the class club-shooting forcing, which enables one to shoot a class club through any fat-stationary class of ordinals. Generalizing the notion from sets, a class $S\of\Ord$ is \emph{fat stationary}, if for every class club $C\of\Ord$ and ordinal $\beta$, there is a closed copy of $\beta+1$ inside $S\intersect C$. The \emph{class club-shooting} forcing $\Q_S$, for a fat-stationary class $S$, is the partial order of all closed sets $c\of S$, ordered by end-extension. Let us now establish the basic properties of this forcing.\goodbreak

\begin{lemma}\label{Lemma.Shooting-club-through-fat-stationary-class}
 Assume $\GBC$, and suppose that $S\of\Ord$ is a fat-stationary class of ordinals. Then the class-club-shooting forcing $\Q_S$ has the following properties:
 \begin{enumerate}
  \item $\Q_S$ is $\less\ORD$-distributive;
  \item it is therefore tame, admits a forcing relation and does not add sets;
  \item it adds a class club $C\subseteq S$;
  \item it preserves the fat-stationarity of any fat stationary subclass $T\of S$; and
  \item it is locally homogeneous.
 \end{enumerate}
\end{lemma}

\begin{proof}
Assume $\GBC$, and suppose that $S\of\Ord$ is a fat-stationary class of ordinals. Let $\Q_S$ be the class-club-shooting forcing for $S$. The fat stationarity of $S$ ensures that any condition in $\Q_S$ can be extended to a longer condition, with order type as long as desired, and so the union of the generic filter will be a class club $C\of S$.

It is easy to see that $\Q_S$ is locally homogeneous: if $c$ and $d$ are conditions in $\Q_S$, then extend them to conditions $c^+$ and $d^+$ having the same supremum, and consider the map $\pi$ defined on extensions of $c^+$ that simply replaces the $c^+$-initial segment of a condition with $d^+$. This is an isomorphism of $\Q_S\restrict c^+$ with $\Q_S\restrict d^+$, verifying local homogeneity.

To prove that the forcing is ${<}\Ord$-distributive, consider any condition $c_0\in \Q_S$ and any coded $\beta$-sequence $\vec D=\<D_\alpha\mid\alpha<\beta>$ of open dense classes $D_\alpha\of\Q_S$. We seek an extension of $c_0$ that is simultaneously in all $D_\alpha$. To find such an extension, apply the reflection theorem to find a class club $E$ of ordinals $\theta$ for which $\langle V_\theta,\in,S\intersect\theta,\vec D\restrict\theta\rangle\prec_{\Sigma_2}\langle V,\in,S,\vec D\rangle$. These structures for $\theta\in E$ therefore form a continuous $\Sigma_2$-elementary chain. Since $S$ is fat stationary, there is a continuous $(\beta+1)$-sequence $\<\theta_\alpha\mid\alpha\leq\beta>$ inside $E\intersect S$, as high as we like. We may assume $c_0\in V_{\theta_0}$ and that $\beta<\theta_0$.

We construct a descending sequence of conditions $c_0\geq c_1\geq\cdots\geq c_\beta$ below $c_0$, ensuring that $c_\alpha\in V_{\theta_{\alpha+1}}$, that $\theta_\alpha\in c_{\alpha+1}$ and that $c_{\alpha+1}\in D_\alpha$ for all $\alpha\leq\beta$. Given $c_\alpha$, add the point $\theta_\alpha$, which is in $S$ and therefore allowed, and extend $c_\alpha\union\singleton{\theta_\alpha}$ to a condition $c_{\alpha+1}\in D_\alpha$ inside $V_{\theta_{\alpha+1}}$, which is possible because this is $\Sigma_2$-elementary in $\langle V,\in,S,\vec D\rangle$. At a limit stage $\lambda$, let $c_\lambda=\Union_{\alpha<\lambda} c_\alpha\union\singleton{\theta_\lambda}$, which simply takes the union of the earlier conditions and adds the supremum $\theta_\lambda$. This is a legal condition because all $\theta_\alpha$ are in $S$. After $\beta$ steps of this construction, the condition $c_\beta$ is an extension of $c_0$ that is inside all the open dense classes $D_\alpha$, and so we have verified distributivity.

It follows that the forcing $\Q_S$ is tame, and therefore has a forcing relation, and from this, it follows easily from $\less\Ord$-distributivity that the forcing $\Q_S$ adds no new sets, as we explained in the previous section.

What remains is to show that the forcing preserves the fat-stationarity of any fat stationary subclass $T\of S$. Suppose that $T\of S$ is fat stationary. We want to see that $T$ remains fat stationary in the forcing extension ${\mathcal V}[C]$. To see this, suppose that $\dot D$ is a $\Q_S$-name for a class club, and fix any ordinal $\beta$. Let $\forces$ be a class forcing relation for $\Q_S$ for $\Sigma^0_2$ assertions in the forcing language with parameters for $T$, $S$, and $\dot D$. Fixing an arbitrary condition $d_0\in \Q_S$, we will argue that there is a stronger condition forcing that $\dot D\cap \check T$ has a closed subset of order-type $\beta+1$. As above, fix a continuous $\Sigma_2$-elementary $(\omega\beta+1)$-chain of structures
 $$\langle V_\theta,\in,T\intersect\theta,S\intersect\theta,\dot D\intersect V_\theta,\forces\restrict V_\theta\rangle\prec_{\Sigma_2}\langle V,\in,T,S,\dot D,\forces\rangle.$$
We build a descending sequence of conditions $d_0\geq d_1\geq\cdots\geq d_{\omega\beta+1}$, such that $d_\alpha\in V_{\theta_{\alpha+1}}$ and $\theta_\alpha\in d_{\alpha+1}$ and $d_{\alpha+1}$ forces another specific ordinal into $\dot D$ above $\theta_\alpha$. By $\Sigma_2$-elementarity, there is such an element below $\theta_{\alpha+1}$. It follows that the final limit condition $d=d_{\omega\beta+1}$ forces that $\dot D$ is unbounded in every $\theta_{\omega\alpha}$ for $\alpha\leq\beta$ and consequently that  $\set{\theta_{\omega\alpha}\mid \alpha\leq\beta}$ is a closed $(\beta+1)$-sequence in the intersection $\dot D\intersect \check T$. Thus, $d$ forces this instance of fat stationarity, and so $T$ remains fat stationary in ${\mathcal V}[C]$, as desired.
\end{proof}

The class club-shooting forcing will be central to the proof of the main theorem. In that context, our fat-stationary class $S$ will itself be Cohen generic, a case for which one can mount an easier direct proof of distributivity. Meanwhile, we thought it worthwhile to provide the general analysis here.
\section{The class Fodor principle is not provable in \KM}

The main theorem follows from the results proved in this section.
\goodbreak

\begin{theorem}\label{Theorem.Main-technical-result}
Every countable model of Kelley-Morse set theory $\mathcal V=\<V,\in,\mathscr C>\models\KM$ has an extension $\mathcal V^+=\<V,\in,\mathscr C^+>\models \KM$, adding new classes but no new sets, with a regressive class function $F:\Ord\to\omega$ that is not constant on any stationary class.
\end{theorem}

\begin{proof}
Assume that $\mathcal V=\<V,\in,\mathscr C>$ is a countable model of \KM. We shall perform forcing in several steps and then take a symmetric submodel of the final model.\footnote{In general symmetric extensions are defined by a group of automorphisms and filters of subgroups. In this paper we can give our models a straightforward description without using the whole machinery, and so we prefer to avoid introducing this technical tool. See \cite{GitmanHamkins:Kelley-MorseSetTheoryAndChoicePrinciplesForClasses} for an in-depth analysis of the symmetric models approach.} For the first step of forcing, let $\P$ be the forcing to add a $\mathcal V$-generic Cohen function, whose conditions are functions $f:\alpha\to\omega$ for some ordinal $\alpha$ ordered by extension. Let $F:\Ord\to\omega$ be $\mathcal V$-generic for $\P$.
Since the forcing $\P$ is $\kappa$-closed for every cardinal $\kappa$, it follows that it is tame and adds no new sets. Consider the corresponding forcing extension $\mathcal V[F]=\<V,\in,\mathscr C[F]>\models\KM$.

For each $n<\omega$, let $A_n=\set{\gamma\in\Ord\mid F(\gamma)=n}$ be the pre-image of $n$. We claim that this is a fat-stationary class in the extension $\mathcal V[F]$. To see this, consider any ordinal $\beta$ and suppose that $\dot C$ names a class club. We may strengthen any condition $f_0\in\P$ to force another specific element into $\dot C$ above the domain of $f_0$, and by iterating this $\omega$-many times, we may find a condition $f_1$ extending $f_0$ which forces that $\dot C$ is unbounded in its domain. If $\gamma=\dom(f_1)$, we are now free to extend further to a condition $f_1^+$ for which $f_1^+(\gamma)=n$. This condition therefore forces that $\gamma$ is in $A_n\intersect\dot C$. By now iterating this process $\beta$ many times, we find a condition $f_\beta^+$ that forces a closed copy of $\beta+1$ into $A_n\intersect\dot C$. So $A_n$ is fat stationary in $\mathcal V[F]$.

Let $B_n=\set{\gamma\in\Ord\mid F(\gamma)>n}$, which is the same as $\Union_{m>n} A_m$. Thus, each $B_n$ is the union of fat-stationary classes and consequently fat stationary itself. In $\mathcal V[F]$, let $\Q_n$ be the class-club-shooting forcing for $B_n$. Now let $\Q=\prod_{n<\omega}\Q_n$ be the full-support product in $\mathcal V[F]$ of all these club-shooting forcing notions.\footnote{Everything in the following argument would have worked just as well had we used finite-support instead of full-support.}

For any $n<\omega$, let $\Q\restrict n=\prod_{m\leq n}\Q_m$ be the restriction of the $\omega$-length product $\Q$ to the first $n+1$-many coordinates. This forcing can be seen as first shooting a club into the fat-stationary class $B_0$, and then into $B_1$, and so on, up to $B_n$. By lemma \ref{Lemma.Shooting-club-through-fat-stationary-class}, each of these classes remains fat-stationary in the next extension, since $B_0\fo B_1\fo\cdots\fo B_n$, and so by induction the product $\Q\restrict n$ is $\less\Ord$-distributive, tame, and adds no sets. (We could alternatively argue directly that $\P*\dot\Q\restrict n$ has a dense subclass that is $\kappa$-closed for every $\kappa$.)

Suppose that $G\of\Q$ is $\mathcal V[F]$-generic for the full product forcing, and consider the resulting extension $\mathcal V[F][G]$. This will not be a model of \KM, nor even of \GBC, since it has the sequence of clubs $C_n\of B_n$, added by $\Q$, as a coded sequence, but the intersection $\Intersect_n C_n$ must be empty, since any $\gamma$ in that intersection has no possible value for $F(\gamma)$.

Our desired final model, however, will not be $\mathcal V[F][G]$, but rather a certain symmetric submodel of this model, which we shall argue is a model of \KM. Namely, let $\mathscr C^+$ consist of the classes added by some stage of the forcing $\Q$, that is, the classes in $\mathcal V[F][G\restrict n]$, for some $n<\omega$. Our desired final model is $\mathcal V^+=\<V,\in,\mathscr C^+>$.

We claim that $\mathcal V^+$ is a model of \KM. Notice first that each of the finite-product models $\mathcal V[F][G\restrict n]$ is a model of \KM\ with the same sets as $V$, because $\P*\Q\restrict n$, as we argued above, is ${<}\ORD$-distributive. It follows that $\mathcal V^+=\<V,\in,\mathscr C^+>$, being the union of a chain of \KM\ models, all with the same first-order part, is a model at least of \GBC. This is because any finitely many classes of $\mathscr C^+$ exist together in some finite extension $\mathcal V[F][G\restrict n]$, where we get all the desired instances of first-order class comprehension using those class parameters.

It remains to see that $\mathcal V^+$ is fully a model of \KM. For this, we need to prove the second-order class comprehension scheme. Suppose that $Z\in\mathscr C^+$ and consider the class $\set{x\mid \mathcal V^+\models\varphi(x,Z)}$, where $\varphi$ may have second-order class quantifiers. The parameter $Z$ exists in some $\mathcal V[F][G\restrict n]$. Since this is a model of \KM, the forcing relation for the rest of the product forcing $\Q^n=\prod_{m\geq n}\Q_m$, augmented with the unary predicate $\check{\mathscr C}$ as discussed in section~\ref{section.second-order-forcing}, is definable in $\mathcal V[F][G\restrict n]$. Using the predicate $\check{\mathscr C}$ we can express in the augmented forcing language that a class $X$ is in $\mathscr C^+$ by saying that $X$ is the interpretation of an element of $\check{\mathscr C}$ by the restriction of the generic filter to the first $m$-many coordinates for some $m<\omega$. Furthermore, the product forcing $\Q^n$ is locally homogeneous in $\mathcal V[F][G\restrict n]$ as witnessed by taking products of isomorphisms that are constructed as in the proof of lemma~\ref{Lemma.Shooting-club-through-fat-stationary-class}. Let's argue that the question of whether $\varphi(x,Z)$ holds in $\mathcal V^+=\<V,\in,\mathscr C^+>$ is settled by all conditions in the same way. If there were conditions $p$ and $p'$ such that $p\forces \varphi(\check a,\check Z)$, but $p'\forces \neg\varphi(\check a,\check Z)$, then we could find conditions $q\leq p$ and $q'\leq q$ and a coordinate-wise defined automorphism $\pi$ of the forcing, as above, taking $q$ to $q'$. Now take a generic extension $\mathcal V[H]$ with $q\in H$. Observe first that $\mathcal V[H\restrict n]=\mathcal V[\pi(H)\restrict n]$ since the automorphism $\pi$ acts coordinate-wise. Thus, the definition of $\mathscr C^+$ that results from using $H$ is the same as that yielded by using $\pi(H)$. But then both the statements forced by $q$ and $q'$ regarding $\varphi$ must hold true in $\mathcal V[H]=\mathcal V[\pi(H)]$, which yields the desired contradiction. It follows that the class $\set{x\mid \mathcal V^+\satisfies\varphi(x,Z)}$ exists already as a class in $\mathcal V[F][G\restrict n]$, definable there using the forcing relation, and consequently this class is in $\mathscr C^+$. So $\mathcal V^+$ is a model of \KM.

Finally, notice that in this model, each $B_n$ contains a class club, and furthermore, the sequence of $B_n$'s exists as a coded class $B=\set{(n,\gamma)\mid \gamma\in B_n}$, since this class is definable directly from $F$. Meanwhile, $\Intersect_{n<\omega} B_n$ is empty, since for any ordinal $\gamma$ we have $\gamma\notin B_n$ when $n=F(\gamma)$. \end{proof}

It is an immediate corollary of the proof of theorem~\ref{Theorem.Main-technical-result} that the class club filter may not be $\sigma$-closed in a model of $\KM$.
\begin{corollary}
Kelley-Morse set theory \KM, if consistent, does not prove that the class club filter is $\sigma$-closed. It is relatively consistent with \KM\ that there is a coded sequence $\<B_n\mid n<\omega>$ of classes, each in the club filter, but the intersection $\Intersect_n B_n=\emptyset$ is empty.
\end{corollary}




\begin{corollary}
 The theory \KM\ plus the negation of the class Fodor principle, and hence also the negation of $\CC$ and even $\CC_\omega(\Pi^0_2)$, is conservative over \KM\ for first-order assertions about sets.
\end{corollary}

\begin{proof}
Theorem \ref{Theorem.Main-technical-result} shows that every countable model of \KM\ has an extension, with the same first-order part, in which the class Fodor principle fails. So any first-order statement failing in a model of \KM\ also fails in a model of \KM\ plus the negation of the class Fodor principle.
\end{proof}


\section{Separating the class Fodor principles}\label{sec:separatingPrinciples}

In this section, we shall separate the various Fodor principles we defined in the introduction.

\begin{theorem}
Every countable model $\mathcal V=\<V,\in,\mathscr C>\models\KM+\CC$ has an extension to a model $\mathcal V^+=\<V,\in,\mathscr C^+>\models \KM$, with new classes but no new sets, in which the class $\less\Ord$-Fodor principle holds and indeed $\CC_\kappa$ holds for every cardinal $\kappa$, but the class Fodor principle and the bounded class Fodor principle fail.
\end{theorem}

\begin{proof}
Assume that $\mathcal V=\<V,\in,\mathscr C>$ is a model of $\KM+\CC$. We shall follow the proof of theorem~\ref{Theorem.Main-technical-result}, but here we begin by adding a generic regressive function $F:\Ord\to\Ord$, rather than $F:\Ord\to\omega$. Let $\P$ be the class forcing to add such a function $F$ via bounded conditions, and consider the forcing extension $\mathcal V[F]=\<V,\in,\mathscr C[F]>$. Since this forcing is tame, this is a model of $\KM+\CC$ by theorem~\ref{theorem.KM+CC-is-preserved-to-forcing-extensions}.

For any ordinal $\xi$, let $B_\xi=\set{\gamma\mid F(\gamma)>\xi}$, in analogy with the classes $B_n$ in the proof of theorem~\ref{Theorem.Main-technical-result}, and define $\Q_\xi$ to be the class-club-shooting forcing for $B_\xi$. Consider the set-support product $\Q=\Pi_{\xi\in\ORD}\Q_\xi$, and let $\Q\restrict\alpha=\Pi_{\xi\leq\alpha}\Q_\xi$ be initial segments of $\Q$. Let $\dot \Q$ be a $\P$-name for $\Q$.

We argue that $\P*\dot\Q\restrict \alpha$ has a dense subclass that is $\kappa$-closed for every cardinal $\kappa$. Specifically, let $\mathbb D$ be the subclass of $\P*\dot\Q\restrict \alpha$ consisting of conditions that are pairs $(p,\check q)$, such that all the closed sets appearing in $q$ have the same supremum $\gamma$, which is also the maximal element in the domain of $p$, and such that $p(\gamma)>\alpha$. This is easily seen to be dense, since any condition $(p,\dot q)$ can be extended to a condition with a check name in the second coordinate, since the forcing $\P$  does not add sets, and then one can extend all the closed bounded sets appearing in $q$ to have the same supremum and arrange such a $\gamma$ and $p(\gamma)$ as desired. This class is $\kappa$-closed for every $\kappa$, simply by taking unions in each coordinate and defining the value of $p$ at the new maximal element to be above $\alpha$. Thus, $\P*\dot\Q\restrict \alpha$ is $\less\Ord$-distributive and consequently tame, and it adds no new sets.

Let $G\subseteq \Q$ be $\mathcal V[F]$-generic and consider the submodel of $\mathcal V[F][G]$ of the form $\mathcal V^+=\<V,\in,\mathscr C^+>$, where $\mathscr C^+$ are the classes added by some stage of the forcing $\Q$, that is, the classes in $V[F][G\restrict\alpha]$, for some ordinal $\alpha$. Analogous arguments to those given in the proof of theorem~\ref{Theorem.Main-technical-result} show that $\mathcal V^+$ satisfies $\KM$ and that the bounded class Fodor principle fails there, since the function $F$ is not bounded on any stationary proper class. Consequently, the class Fodor principle also fails.

It remains to show that $\CC_\kappa$ holds in $\mathcal V^+$ for every cardinal $\kappa$, which by theorem~\ref{Theorem.CC-implies-Fodor} implies $\less\Ord$-Fodor. Suppose that $\mathcal V^+\models\forall\alpha{<}\kappa\,\exists X\,\varphi(\alpha,X,Z)$ with a class parameter $Z$. Since $Z\in \mathcal V^+$, there must be a stage $\gamma$ such that $Z\in \mathcal V[F][G\restrict\gamma]$, which is a model of $\KM+\CC$ by theorem~\ref{theorem.KM+CC-is-preserved-to-forcing-extensions}. Consequently, $\mathcal V[F][G\restrict\gamma]$ has a forcing relation for the tail forcing $\Q^\gamma$ in the language with the ground-model class predicate $\check{\mathscr C}$. So there is some condition $p\in G^\gamma$ (the generic filter for the tail forcing $\Q^\gamma$ derived from $G$) forcing that for every $\alpha<\check\kappa$, there is a class $X$ in $\mathcal V^+$ such that $\varphi^{\mathcal V^+}(\alpha,X,\check Z)$.  For each $\alpha<\kappa$, the class of conditions $q$ deciding that there is a witness class $X$ added by some specific least stage $\beta$ is dense below $p$. Since the forcing $\Q^\gamma$ is locally homogeneous, all such conditions $q$ must agree on this ordinal $\beta=\beta_\alpha$, the least stage by which $\alpha$ gets its witness class $X$. Since $\mathcal V[F][G\restrict\gamma]\models\KM$, the sequence of $\beta_\alpha$ for $\alpha<\kappa$ is bounded in the ordinals, and if $\beta=\sup_\alpha\beta_\alpha$, then $p$ forces that every $\alpha<\kappa$ has a witness class $X$ in $\mathcal V[F][G\restrict \beta]$ satisfying $\varphi(\alpha,X,Z)$ in $\mathcal V^+$. We now switch ground models and move to $\mathcal V[F][G\restrict\beta]$. In $\mathcal V[F][G\restrict\beta]$, fix an ordinal $\alpha$ and a class $X$ such that $\mathcal V^+\models\varphi(\alpha,X,Z)$. There must be some condition $q\in\Q^\beta$ forcing that $\mathcal V^+\models\varphi(\alpha,\check X,\check Z)$, but then by local homogeneity of $\Q^\beta$, the statement must be forced by all other conditions as well. Now recall that $\mathcal V[F][G\restrict\beta]$ is a model of $\KM+\CC$ by theorem~\ref{theorem.KM+CC-is-preserved-to-forcing-extensions}, and hence it can collect for every $\alpha<\kappa$, a witnessing class $B_\alpha$ such that $\mathcal V^+\models\varphi(\alpha,B_\alpha,Z)$ into a single class $B$. By definition of $\mathcal V^+$, we have $B\in\mathscr C^+$. Thus, we have verified that $\CC_\kappa$ holds in $\mathcal V^+$.
\end{proof}

Essentially the same arguments give the following.

\begin{theorem}
Suppose $\mathcal V=\<V,\in,\mathscr C>$ is a countable model of $\KM+\CC$ and $\kappa$ is a regular cardinal in $V$. Then $\mathcal V$ has an extension $\mathcal V^+=\<V,\in,\mathscr C^+>\models \KM$ in which $\CC_\lambda$ holds for every $\lambda<\kappa$, so, in particular, the $\lambda$-class Fodor principle holds, but the class $\kappa$-Fodor principle fails.
\end{theorem}

Finally, we separate the bounded class Fodor principle from the class $\omega$-Fodor principle.
\begin{theorem}
Every countable model $\mathcal V=\<V,\in,\mathscr C>\models\KM+\CC$ has an extension to a model $\mathcal V^+=\<V,\in,\mathscr C^+>\models \KM$, adding new classes but no new sets, in which the bounded class Fodor principle holds, but the class $\omega$-Fodor principle fails.
\end{theorem}

\begin{proof}
We use the model $\mathcal V^+=\<V,\in,\mathscr C^+>$ from the proof of theorem~\ref{Theorem.Main-technical-result}, and also use all the notation from that proof. The difference is that we now also assume \CC\ in $\mathcal V$. It suffices to show that the bounded class Fodor principle holds in $\mathcal V^+$. Let $H:S\to \Ord$ be a regressive function on a stationary class $S$ in $\mathcal V^+$. The class $H$ is added by some stage $\mathcal V[F][G\restrict m]$ for some $m<\omega$. Let $S_\alpha=\set{\gamma\mid H(\gamma)<\alpha}$, and note that $\alpha<\beta\implies S_\alpha\subseteq S_\beta$. If no $S_\alpha$ is stationary in $\mathcal V^+$, then for each $\alpha$ there is stage $n=n_\alpha$ at which $S_\alpha$ becomes non-stationary in $\mathcal V[F][G\restrict n]$. Since the tail forcing $\Q^m$ is locally homogeneous, this $n_\alpha$ does not depend on $G$ and the function $\alpha\mapsto n_\alpha$ exists already at stage $m$. So there must be some $n$ occurring for unboundedly many $\alpha$. So unboundedly many $S_\alpha$ and hence all $S_\alpha$ are nonstationary already at this stage $\mathcal V[F][G\restrict n]$. But this is a model of \CC, and consequently satisfies the class Fodor theorem, contradicting that $H$ is not bounded on any stationary class in this model.
\end{proof}

\section{The class Fodor principle is invariant by set forcing}

In this section, we shall show that the class Fodor principle is preserved by set forcing and that set forcing cannot make the class Fodor principle hold in a forcing extension if it did not already hold in the ground model. The main preliminary fact is that after set-sized forcing, every new class club contains a ground model class club. This is because if a condition $p\forces\dot C$ is a class club, then the class of ordinals $\alpha$ such that $p\forces\check\alpha\in\dot C$ is closed and unbounded, by the class analogue of the usual argument for $\kappa$-c.c. forcing and club subsets of $\kappa$.

\begin{theorem}\label{Theorem.Fodor-invariant-by-set-forcing}
 The class Fodor principle is invariant by set forcing over models of $\GBC^-$. That is, it holds in an extension if and only if it holds in the ground model.
\end{theorem}

\begin{proof}
Assume $\GBC^-$, and consider a set-forcing extension $\mathcal V[G]=\<V[G],\in,\mathscr C[G]>$ of $\mathcal V=\<V,\in,\mathscr C>$, where $G\of\P\in V$ is $\mathcal V$-generic.

Suppose first that the class Fodor principle holds in the ground model $\mathcal V$, and suppose that $F:S\to\Ord$ is a class function in the extension $\mathcal V[G]$ that is regressive on the stationary class $S\of\Ord$ in $\mathscr C[G]$. Let $\dot S$ and $\dot F$ be class $\P$-names for which $S=\dot S_G$ and $F=\dot F_G$. In the ground model $\mathcal V$, for each ordinal $\alpha$ choose a condition $p_\alpha\in\P$ forcing $\check\alpha\in\dot S$, if possible, and deciding the value of $\dot F(\check\alpha)$. Since $S$ is stationary in the extension and therefore meets all ground-model class clubs there, it follows that $p_\alpha$ is defined on a stationary class of ordinals. Since $\P$ is a set, it follows by the class Fodor principle that $p_\alpha=p$ is constant on a stationary class $S_0$ of ordinals. For $\alpha\in S_0$, let $E(\alpha)$ be the value of $\dot F(\check\alpha)$ that is forced by $p$. This is a regressive function on $S_0$ in the ground model, and so by the class Fodor principle, it is constant on a stationary subclass $S_1\of S_0$. So $p$ forces that $\dot F$ is constant on $\check S_1$, which remains a stationary class in $\mathcal V[G]$, since as we observed before the theorem, every new class club contains a ground-model class club. So $p$ forces this instance of the class Fodor principle in $\mathcal V[G]$. In particular, by relativizing below any given condition, there can be no condition forcing a violation of the class Fodor principle in $\mathcal V[G]$, and so it must hold there, as desired.

Conversely, assume that the class Fodor principle holds in a set-forcing extension $\mathcal V[G]=\<V[G],\in,\mathscr C[G]>$ of $\mathcal V=\<V,\in,\mathscr C>$, and that $F:S\to\Ord$ is a regressive class function on a stationary class $S$ in the ground model. Since every new class club contains a ground model class club, $S$ remains stationary in $\mathcal V[G]$, and so by the class Fodor principle in $\mathcal V[G]$, the function $F$ is constant on a stationary class in $\mathcal V[G]$. So $F^{-1}(\singleton{\alpha})$ is a stationary class in $\mathcal V[G]$ for some ordinal $\alpha$. But this class exists already in the ground model, and therefore must be stationary there. So the class Fodor principle holds in $\mathcal V$, as desired.
\end{proof}

The downward absoluteness of the class Fodor principle can be easily generalized to $\Ord$-c.c.~class forcing, for which also every new class club contains a ground model class club.

\begin{question}\label{Queston.Fodor-invariant-by-class-forcing?}
 Is the class Fodor principle invariant by pre-tame class forcing? Is it preserved by such forcing? By $\Ord$-c.c.~class forcing?
\end{question}

\section{Fodor on higher class well orders}

Let us consider an analogue of Fodor on class well-orders beyond $\Ord$. Namely, if $\Gamma$ is a class well-order of uncountable cofinality, then we have sensible notions of \emph{club in $\Gamma$} and \emph{stationary in $\Gamma$}, as well as notions of \emph{regressive} functions on $\Gamma$.

First, let's notice that class well-orders have cofinality at most $\Ord$.

\begin{lemma}\label{Lemma.Cofinality-of-class-orders}
Assume \GBC. Every class well-order $\Gamma$ contains a closed unbounded suborder of type $\Ord$ or type $\kappa$ for some regular cardinal $\kappa$.
\end{lemma}

\begin{proof}
Suppose that $\Gamma$ is a class well-order. If some set is unbounded in the order $\Gamma$, then we can find a cofinal $\kappa$-sequence for some regular cardinal $\kappa$. So assume that no set is unbounded in the order $\Gamma$. In this case, let $a_0$ be the least element of $\Gamma$, and define $a_{\alpha+1}$ to be the $\Gamma$-least element above $a_\alpha$ and above all elements of $V_\alpha$, and $a_\lambda$ to be the supremum of $a_\alpha$ for $\alpha<\lambda$, when $\lambda$ is a limit ordinal. It follows that the class $\set{a_\alpha\mid\alpha\in\Ord}$ is an unbounded subclass of $\Gamma$ of order-type $\Ord$.
\end{proof}

Thus, every class well-order $\Gamma$ has a \emph{cofinality} $\kappa\leq\Ord$, which is the smallest order-type of any closed unbounded subclass.

Let us briefly note that in \GB, even without the global axiom of choice (or \AC\ for sets), one may always take the underlying class of a proper class well-order $\Gamma$ to be $\Ord$. The reason is that if $\Gamma=\<X,\triangleleft>$ is a proper class well-order, then the underlying class $X$ is stratified by $\in$-rank. So we can enumerate the class $X$ in a set-like well-ordered manner, enumerating first by rank and then by the $\Gamma$-order within each rank. This provides a bijection of $X$ with $\Ord$, and so $\Gamma$ is isomorphic to a relation on $\Ord$.

\begin{theorem}
 Assume \GBC\ and the class Fodor principle. Then for any class well-order $\Gamma$, every regressive class function $F:S\to\Gamma$ defined on a stationary subclass $S\of\Gamma$ is bounded on a stationary subclass $T\of S$.
\end{theorem}

\begin{proof}
Suppose that $F: S\to\Gamma$ is regressive on a stationary class $S\of\Gamma$. By lemma \ref{Lemma.Cofinality-of-class-orders}, there is a closed unbounded subclass $C\of\Gamma$ of order type $\kappa\leq\Ord$. It is easy to see that $S\intersect C$ remains stationary in $\Gamma$. Define a class function $F_0:C\to C$ $F_0(a)=\sup\set{b\in C\mid b\leq F(a)}$, which presses the value of $F(a)$ down to $C$. The class function $F_0$ is regressive on $C$, and is therefore isomorphic to a regressive function on $\kappa$. By the class Fodor principle (needed only for the case $\kappa=\Ord$), it follows that $F_0\restrict T$ is constant for some stationary subclass $T\of S\intersect C$. The class $T$ remains stationary in $\Gamma$, and it follows that $F(a)$ is pressed down always to the same element of $C$ for $a\in T$. Thus, $F(a)$ is bounded by the next element of $C$ for $a\in T$, and so $F$ is bounded on a stationary subclass of $S$.
\end{proof}

One cannot expect in general that regressive functions on class well-orders exceeding $\Ord$ will be constant on a stationary subclass, since if $\Gamma$ is a class well-order exceeding $\Ord$, then by lemma \ref{Lemma.Cofinality-of-class-orders}, there is a closed unbounded subclass $C\of\Gamma$ of order-type at most $\Ord$, and we can define $F(\gamma)$ to be the $\Gamma$-order type of the $\Gamma$-predecessors of $\gamma$ in $C$, that is, the set $\set{\alpha\in C\mid \alpha<_\Gamma\gamma}$. This order-type is strictly less than $\Ord$, and so the function $F$ is regressive on the part of $\Gamma$ beyond $\Ord$, but it cannot be constant on any unbounded subclass of $\Gamma$.

\section{The strength of the class Fodor principle}\label{Section.Questions}
It follows from a result of Ali Enayat about expansions of models of $\ZFC$ to models of $\GBC+\CC(\Sigma^1_1)$ that the theory $\GBC$ together with the class Fodor principle is equiconsistent with $\ZFC$.
\begin{theorem}[Enayat \cite{Enayat:BarwiseSchlipfSetTheory}]
Every countable recursively saturated model of $\ZFC$ has an expansion to a model of $\GBC+\CC(\Sigma^1_1)$ with the same sets. \footnote{If the $\ZFC$-model has a definable global well-order, then its definable classes satisfy $\GBC+\CC(\Sigma^1_1)$. Otherwise, we force with $\Add(\Ord,1)$ to add a global well-order class and the classes generated by the generic global well-order satisfy $\GBC+\CC(\Sigma^1_1)$.}
\end{theorem}
A model of $\ZFC$ is said to be \emph{recursively saturated} if it realizes every finitely realizable recursive type $p(x,a)$ with parameter $a$, where a type $p(x,y)$ is \emph{recursive} if the collection of \Godel-codes of formulas in the type is a recursive set. By closing under realizations of recursive types in $\omega$-many steps, it can be shown that every countable model of $\ZFC$ has a countable recursively saturated elementary extension.
\begin{corollary}\label{cor:GBCEquiconsistent}
The theory $\GBC$ together with the class Fodor principle is equiconsistent with $\ZFC$ and is conservative over $\ZFC$ for all first-order assertions.
\end{corollary}
\begin{proof}
This follows immediately from Enayat's theorem by observing that every countable model of $\ZFC$ has a countable recursively saturated elementary extension, and the elementary extension has an expansion to a model of $\GBC+\CC(\Sigma^1_1)$, and so, in particular, the Fodor principle holds there.
\end{proof}
In a further work it would be interesting to study in which natural models of $\ZFC$ or $\GBC$ the class Fodor principle holds or fails.

\begin{question}
Suppose $0^\#$ exists, and consider the \GBC\ model $\<L,\in,\text{\rm Def}(L)>$, where $\text{\rm Def}(L)$ has only the first-order definable classes with set parameters. Does this model satisfy the class Fodor principle?
\end{question}

This model has some weak forms of Fodor. Namely, if $F:S\to\ORD$ is definable from the indiscernible parameters $i_{\xi_0},\ldots,i_{\xi_n}$ and some indiscernible $i_\xi$ with $\xi>\xi_n$ is in $S$, then $F$ is constant on a stationary subclass. In this case, all indiscernibles above $i_\xi$ must also be in $S$, and moreover, $F$ has the same value on all of them. But there are stationary classes in this model, such as $\Cof_\omega^L$, which do not contain any indiscernibles, so that observation does not settle the question.

The question is related to stationary reflection. Namely, if $S\intersect\kappa$ is stationary in $\kappa$ for some large enough $\kappa$, where $L_\kappa\elesub L$, then consider $F\restrict\kappa$. There must be $\alpha<\kappa$ with $F^{-1}(\singleton{\alpha})$ actually stationary in $\kappa$, hence definably stationary in $\kappa$. Apply elementarity. So $F$ has constant value $\alpha$ on a stationary class.

\newcommand{\etalchar}[1]{$^{#1}$}


\begin{thebibliography}{GHH{\etalchar{+}}17}

\bibitem[AFG]{AntosFriedmanGitman:booleanclassforcing}
Carolin Antos, Sy~David Friedman, and Victoria Gitman.
\newblock Boolean-valued class forcing.
\newblock {\em manuscript under review}.

\bibitem[Ant18]{antos:ClassForcingInClassSetTheory}
Carolin Antos.
\newblock Class forcing in class theory.
\newblock In {\em The hyperuniverse project and maximality}, pages 1--16.
  Birkh\"auser/Springer, 2018.

\bibitem[Ena]{Enayat:BarwiseSchlipfSetTheory}
Ali Enayat.
\newblock A {B}arwise-{S}chlipf theorem for set theory.
\newblock {\em In progress}.

\bibitem[Fri00]{friedman:classforcing}
Sy~D. Friedman.
\newblock {\em Fine structure and class forcing}, volume~3 of {\em de Gruyter
  Series in Logic and its Applications}.
\newblock Walter de Gruyter \& Co., Berlin, 2000.

\bibitem[GH]{GitmanHamkins:Kelley-MorseSetTheoryAndChoicePrinciplesForClasses}
Victoria Gitman and Joel~David Hamkins.
\newblock Kelley-{M}orse set theory and choice principles for classes.
\newblock In preparation.

\bibitem[GH16]{GitmanHamkins2016:OpenDeterminacyForClassGames}
Victoria Gitman and Joel~David Hamkins.
\newblock Open determinacy for class games.
\newblock In Andr\'es~E. Caicedo, James Cummings, Peter Koellner, and Paul
  Larson, editors, {\em Foundations of Mathematics, Logic at Harvard, Essays in
  Honor of Hugh Woodin's 60th Birthday}, AMS Contemporary Mathematics. 2016.
\newblock Newton Institute preprint ni15064.

\bibitem[GHH{\etalchar{+}}17]{GitmanHamkinsHolySchlichtWilliams:The-exact-strength-of-the-class-forcing-theorem}
Victoria Gitman, Joel~David Hamkins, Peter Holy, Philipp Schlicht, and Kameryn
  Williams.
\newblock The exact strength of the class forcing theorem.
\newblock {\em ArXiv e-prints}, July 2017.
\newblock manuscript under review.

\bibitem[GHJ16]{GitmanHamkinsJohnstone2016:WhatIsTheTheoryZFC-Powerset?}
Victoria Gitman, Joel~David Hamkins, and Thomas~A. Johnstone.
\newblock What is the theory {ZFC} without {Powerset}?
\newblock {\em Math.~Logic Q.}, 62(4--5):391--406, 2016.

\bibitem[HKS18]{PeterHolyRegulaKrapfPhilippSchlicht:classforcing2}
Peter Holy, Regula Krapf, and Philipp Schlicht.
\newblock Characterizations of pretameness and the {O}rd-cc.
\newblock {\em Ann. Pure Appl. Logic}, 169(8):775--802, 2018.

\bibitem[Kar18]{Karagila2018:Fodors-lemma-can-fail-everywhere}
A.~Karagila.
\newblock Fodor's lemma can fail everywhere.
\newblock {\em Acta Math. Hungar.}, 154(1):231--242, 2018.

\bibitem[KS20]{SchlichtKaragila}
Asaf Karagila and Philipp Schlicht.
\newblock How to have more things by forgetting how to count them.
\newblock {\em Proceedings of the Royal Society A: Mathematical, Physical and
  Engineering Sciences}, 476(2239):20190782, 2020.

\bibitem[Mar73]{Marek:KM}
W.~Marek.
\newblock On the metamathematics of impredicative set theory.
\newblock {\em Dissertationes Math. (Rozprawy Mat.)}, 98:40, 1973.

\bibitem[Neu51]{Neumer1951:Verallgemeinerung-eines-Satzes-von-Alexandroff-und-Urysohn}
Walter Neumer.
\newblock Verallgemeinerung eines {S}atzes von {A}lexandroff und {U}rysohn.
\newblock {\em Mathematische Zeitschrift}, 54(3):254--261, Sep 1951.

\bibitem[Wil18]{williams:dissertation}
Kameryn~J. Williams.
\newblock {\em The {S}tructure of {M}odels of {S}econd-{O}rder {S}et
  {T}heories}.
\newblock ProQuest LLC, Ann Arbor, MI, 2018.
\newblock Thesis (Ph.D.)--City University of New York.

\end{thebibliography}
\end{document}